\documentclass[a4paper]{article}
\usepackage{amsmath,amsthm,amsfonts}
\usepackage[dvips]{graphicx}

\usepackage[usenames]{color}
\usepackage{graphicx,type1cm,eso-pic,color}

\newcommand{\R}{\mathbb{R}}
\newcommand{\Z}{\mathbb{Z}}

\newcommand{\N}{\mathbb{N}}

\newcommand{\C}{\mathbb{C}}
\newcommand{\A}{\mathcal{A}}

\newcommand{\TT}{\mathcal{T}}

\newcommand{\M}{\mathcal{M}}

\newcommand{\Iso}{\text{Iso}}
\newcommand{\inj}{\text{inj}}

\newcommand{\e}{\varepsilon}

\newcommand{\diag}{\text{diag}}
\newcommand{\Om}{\Omega}
\newcommand{\al}{\alpha}
\newcommand{\Int}{\text{Int}}

\newcommand{\skp}{\langle .,.\rangle}

\newcommand{\h}{h_{\text{top}}}
\newcommand{\hloc}{h_{\text{top,loc}}}

\DeclareMathOperator{\id}{id}

\newtheorem{definition}{Definition}[section]
\newtheorem{lemma}[definition]{Lemma}
\newtheorem{prop}[definition]{Proposition}
\newtheorem{theorem}[definition]{Theorem}
\newtheorem{cor}[definition]{Corollary}
\newtheorem{remark}[definition]{Remark}

\def\vol{{\rm vol }}

\def\inj{{\rm inj }}
\def\diam{{\rm diam\, }}

\usepackage{hyperref}

\usepackage[nottoc]{tocbibind}

\begin{document}

\title{Topological entropy of minimal geodesics and volume growth on surfaces}

\author{ Gerhard Knieper, Carlos Ogouyandjou, Jan Philipp Schr\"oder
\\Fakult\"at f\"ur
Mathematik\\Ruhr--Universit\"at Bochum\\44780
 Bochum\\Germany\\}

\date{\today}

\maketitle

\begin{abstract}
Let $(M,g)$ be a compact Riemannian manifold of hyperbolic type,
i.e $M$ is a manifold admitting another metric of strictly
negative curvature. In this paper we study the geodesic flow
restricted to the set of geodesics which are
minimal on the universal covering.  In particular for surfaces we
show that the topological entropy of the minimal geodesics coincides with the volume
entropy of $(M,g)$ generalizing work of Freire and Ma\~n\'e.
\end{abstract}

\tableofcontents

\section{Introduction and main results}

Let $(M,g)$ be a compact Riemannian manifold (connected and $\partial M=\emptyset$) and $p:\tilde M\to M$ its universal
Riemannian covering, saving $\pi:TM\to M$ for the canonical projection. In \cite{Ma}, Manning introduced the volume
entropy (also called volume growth) $h(g)$ of $(M,g)$ defined by
$$ h(g):=\lim _{r \rightarrow +\infty }\frac{1}{r}\log \vol B(p,r),$$
where $p\in \tilde M$ and $B(p,r)$ denotes the open ball with center $p$ and
radius $r$. He proved that this limit exists and is independent of
$p$. Let $\h(\phi ^t)=\h(\phi^t_{SM})$ denote the topological entropy of the
geodesic flow $\phi ^t$ on the unit tangent bundle $SM$. Manning
proved the following estimate:
\[ \h(\phi^t_{SM}) \geq h(g) . \]
In the case of nonpositive curvature he showed that
equality holds. Subsequently this was generalized by Freire and
Ma\~n\'e \cite{FM} to  metrics without conjugate points. Let $\tilde \M$ be the closed and $\phi^t$-invarint subset of $S \tilde M$ consisting of all $v\in S\tilde M$ such that
the geodesic $ c_v $ with $\dot c_v(0) = v$ is
globally minimizing. We denote by $\M=Dp(\tilde \M)$ the projection of $\tilde \M$ to $SM$ and by $ \phi ^t_\M,\phi ^t_{\tilde \M}$ the geodesic flow restricted to $\M,\tilde \M$, respectively. In
\cite{KH} Katok and Hasselblatt stated the following theorem, saying that it is enough to consider minimal geodesics to generate exponential complexity (provided $h(g)>0$).

\begin{theorem}
Let $(M,g)$ be a compact Riemannian manifold and $ \phi^t_\M$ be the geodesic flow $\phi ^t$ restricted to the minimal geodesics $\M\subset SM$. Then
$$ 
\h(\phi ^t_\M) \geq h(g).
$$
\end{theorem}

Following Klingenberg \cite{Kl} we call a compact manifold $M$ to be of hyperbolic type, if there exists
a metric of strictly negative curvature $g_0$ on $M$. 
We hope to prove an inequality of the kind $\h(\phi ^t_\M) \leq h(g)$, i.e. that equality holds in the above theorem. A first result in this direction is the following. We will introduce the notation $h_{top}( \phi ^t,F, \beta)$ and the notion of entropy expansiveness in section \ref{def h_top}.

\begin{theorem}
Let $(M,g)$ be a compact Riemannian manifold of hyperbolic type.
There is some constant $\beta$ depending only on $(M,g)$ such that for each compact set
$K \subset \tilde M$ we have
$$ h_{top}( \phi ^t, \pi^{-1}(K) \cap \tilde\M, \beta) \leq h(g). $$
\end{theorem}

Using a result of Bowen \cite{Bo}, which we shall prove below in the non-compact setting, we obtain the following.

\begin{cor} \label{h=h(g) expansive}
Let $(M,g)$ be a compact Riemannian manifold of hyperbolic type. If $\phi ^t_{\tilde \M}$ is $\beta$-entropy-expansive for $\beta$ from the above theorem, we have
$$ \h(\phi ^t_\M) = h(g).$$
\end{cor}

Presently we do not know if $\phi^t_{\tilde \M}$ for Riemannian manifolds $(M,g)$ of hyperbolic type of arbitrary dimension is $\beta$-entropy-expansive. We shall prove, however, that in the two-dimensional case, $\beta$-entropy-expansiveness holds in the non-wandering set of $\M$. This gives the following result.

\begin{theorem}
Let $(M,g)$ be closed Riemannian surface. Then
$$ \h( \phi ^t_\M) = h(g). $$
\end{theorem}

This paper is organized as follows. In the second section we study topological entropy and local topological entropy
for homeomorphisms of metric spaces and following the ideas of Bowen \cite{Bo} we provide an estimate for the topological entropy.
In section 3, we give a complete proof using the ideas provided by Katok and Hasselblatt
that the topological entropy of the minimal geodesics is bounded
below by the volume growth (theorem 1.1). Moreover, we study
topological entropy of minimal geodesics on manifolds of
hyperbolic type and give the proof of theorem 1.2. Finally, in section 4 we show that for surfaces the topological entropy of $\phi^t_\M$ equals the
volume growth of $g$ (theorem 1.4).

\section{Topological Entropy for homeomorphisms of metric spaces}

In this section we study discrete dynamical systems. In order to apply our results to geodesic flows $\phi^t, t\in \R,$ observe that the topological entropy of $\phi^t$ defined in the continuous setting coincides with that of the discrete system $\phi^n, n\in\Z$, cf. \cite{KH}.

\subsection{Bowen's definition}\label{def h_top}

Here we recall Bowen's definition of topological entropy. Let $f : V \to V$ be a homeomorphism of a not necessarily compact metric space $(V,d)$. For each $n\in \N$, a 
 metric on $V$ is defined by
$$
d_n(x,y) := \max_{0 \le j <n} d(f^j(x), f^j(y)) .
$$
Let $F$ be a subset of  $V$. We say that a set $Y\subset V$ is $(n,\e)$-spanning for $F$ if the closed balls 
$\bar B_n(y,\e) = \{y\in V: d_n(x, y) \le \e \}, y\in Y$ cover $F$.
If $Y \subset F$ and $\bar B_n(y,\e)\cap Y = \{y\}$ for all $y\in Y$, we say that $Y$ is an $(n,\e)$-separated subset of $F$.

Let $r_n(F,\e)$ denote the minimal cardinality of $(n,\e)$-spanning sets for $F$ and let $s_n(F,\e)$ denote the maximal cardinality
of $(n,\e)$-separated subsets of $F$. It is easy
to see that for any $\e > 0$ we have
$$
r_n(F,\e ) \leq
s_n(F,\e ) \leq r_n(F,\e /2) .
$$
Note that $r_n(F,\e)<\infty$, if $F$ is compact.

We define the following notions of topological entropy.
\begin{align*}
 \h(f,F, \e) &:= \varlimsup_{n\to+\infty}\frac{1}{n}\log r_n(F,\e), \\
 \h(f,F) &:=\lim_{\e\to 0}\h (f,F, \e) , \\
 \h(f) &:= \sup_{ F \subset V \text{ compact}} \h(f, F).
\end{align*}
Note that for any $\e>0$ we have $\h(f,F, \e) \leq \h(f,F)$ and if $V$ is itself compact, we get $\h(f)=\h(f,V)$. If we use $s_n(F,\e)$ instead of $r_n(F,\e)$, we obtain the same value for $\h(f,F)$. For details on topological entropy we refer to \cite{W}.

We need the following less known
concept of local entropy introduced by Bowen \cite{Bo}.
 For $x\in V$ and $\beta >0$ set
$$Z_\beta(x):=\{y\in V : d(f^n(x), f^n(y))\leq \beta   \; \forall n\in \mathbb{Z}\}.$$
Then we call 
$$
\hloc(f,\beta):=\sup_{x\in V}\h(f, Z_\beta (x))$$
the $\beta$-local entropy of $f$.
We say that $f$ is $\beta$-entropy-expansive for $\beta>0$ if
$$\hloc(f,\beta) =0.$$

\subsection{An upper bound for the topological entropy of homeomorphisms}\label{bowen bounds}

In order to make use of the local entropy it will be important to compute entropy on coverings.
We consider the following setting. 
Let $(\tilde V, \tilde d)$ be a metric space and $\Gamma$ a subgroup of isometries of $\tilde V$ acting
on $\tilde V$. Assume that the quotient $V := \tilde V/\Gamma$ is compact and equipped
with a metric $d$ such that the projection $p: \tilde V \to V$ is a local isometry. Let $ \tilde f: \tilde V \to \tilde V$ be a homeomorphism
which commutes with the group $\Gamma$ and let $f: V \to V$ be the projection defined by $f(x)=p\tilde f p^{-1}(x)$ (this is well-defined since $\tilde f,\Gamma$ commute). $f$ is a homemorphism as well. Recall the following result.

\begin{prop}[theorem 8.12 in \cite{W}] \label{walters}
For each compact set $K \subset  \tilde V$ we have
$$
h_{top}(\tilde f, K) = h_{top}( f, p(K))
$$
In particular, if $p(K)=V$, then
\[ h_{top}(\tilde f, K) = h_{top}( f). \]
\end{prop}

We shall prove the following theorem which is a slight extension of a result of Bowen
(see \cite{Bo}). It allows to estimate the topological entropy using coverings and will be crucial for our applications.

\begin{theorem}\label{bowen expansive}
Let $K \subset \tilde V$ be a compact set such that $p(K) = V$. Then for any
$\beta >0$ we have
$$
\h( f) \le \h( \tilde f, K, \beta) + \hloc( \tilde f,  \beta).
$$
\end{theorem}

The proof of \ref{bowen expansive} rests of the following estimate.

\begin{lemma} \label{lemma bowen}
Let $a = \hloc( \tilde f, \beta)$. For any $\e >0, \delta>0, \beta>0$ there exists a constant $c>0$, s.th.
\[ r_n\left(\bar B_n(x,\beta) ,\delta \right) \leq c e^{(a+\e)n} \quad \forall ~ x\in K, n\in\N. \]
\end{lemma}

We need the following elementary lemma (see \cite{Bo}).

\begin{lemma}\label{lemmabowen1}
Let $F\subset \tilde V$ and consider integers $0 = t_0 < t_1 < ... < t_r =n$. For $\al>0$ and $0\leq i <r$ let $E_i $ be a $(t_{i+1} -t_i, \al)$-spanning set
for $\tilde f^{t_i}(F)$. Then
\[ r_n(F, 2 \alpha) \leq \prod_{i=0}^{r-1} \# E_i . \]
\end{lemma}

\begin{proof}[Proof of \ref{lemmabowen1}]
For $(x_0, \ldots, x_{r-1}) \in E_0 \times \cdots \times E_{r-1}$ set
\begin{align*}
& B(x_0, \ldots, x_{r-1}) \\
& := \{ x \in F \mid ~ d(\tilde f^{t+t_i}(x), \tilde f^t(x_i)) \le \alpha ~~ \forall ~ 0 \le i <r, t \in [0, t_{i+1} -t_i]\cap\Z \}.
\end{align*}
By assumption the $B(x_0,...,x_{r-1})$ cover $F$ and using the triangle inequality we have $d_n(x,y) \le 2 \alpha $ for all $x,y \in B(x_0, ..., x_{r-1})$. Choosing from each nonempty set $B(x_0, \ldots, x_{k-1})$ one element we obtain a $(n,2 \alpha)$-spanning set. This yields the estimate.
\end{proof}

\begin{proof}[Proof of \ref{lemma bowen}]
In the following fix positive numbers $\e,\delta,\beta>0$, a point $x\in K$, an integer $n\in\N$ and set $F:= \bar B_n(x,\beta)$. We shall try to describe the orbit $\{x, \tilde f x, ... , \tilde f^{n-1} x\}$ by a finite collection of $y$'s in $K$ and their sets $Z_\beta(y)$.

\underline{Step 1.} (choice of $y_1,...,y_s\in K$ and appropriate neighborhoods $V(y_i)$) By definition of $a$ we find for all $y\in K$ some integer $m(y)\in\N$ and a $(m(y),\delta/2)$-spanning set $E(y)$ for $Z_\beta(y)$ with
\[ \frac{1}{m(y)} \log \# E(y) \leq a+\e . \]
Define the open neighborhoods
\[ U(y) := \bigcup_{z\in E(y)} B_{m(y)}(z,\delta/2) ~\supset ~ Z_\beta(y), \quad y\in K .\]
For $N\to\infty, R \searrow \beta$ the compact sets 
\[ W_N(y,R) := \bigcap_{|j|\leq N} \tilde f^{-j} \bar B(\tilde f^jy, R)  \]
decrease to the compact set $Z_\beta(y)$, so we find $N(y)\in\N,R(y)>\beta$, s.th. $W_{N(y)}(y,R(y))$ is contained in the neighborhood $U(y)$ of $Z_\beta(y)$. Define
\[V(y) := \Int W_{N(y)}(y,R(y)-\beta) , \quad y\in K. \]
The triangle inequality implies that
\[ (*) \qquad  \forall z\in V(y): \quad W_{N(y)}(z,\beta) \subset W_{N(y)}(y,R(y)) \subset U(y). \]
By the compactness of $K$ we find $y_1,...,y_s\in K$ with
\[ \tilde V = \bigcup_{\gamma\in\Gamma}\bigcup_{i=1}^s \gamma V(y_i). \]
Set
\[n_0 := \max_{1\leq i\leq s}\max \{ N(y_i),m(y_i) \} \in \N. \]

\underline{Step 2.} (describtion of $F$ by the $y_i$'s) We claim the following:
\[ (**) \qquad  \forall t\in [n_0,n-n_0)\cap \Z ~~\exists i\in\{1,...,s\}, \gamma\in\Gamma: \quad \tilde f^t (F) \subset \gamma U(y_i). \]
{\it Proof of the claim}. We find $\gamma,i$ with $\tilde f^tx\in \gamma V(y_i)$, and hence
\begin{align*}
\tilde f^t(F) & = \bigcap_{j=0}^{n-1} \tilde f^{t-j} \bar B(\tilde f^jx,\beta) = \bigcap_{j=-t}^{n-t-1} \tilde f^{-j} \bar B(\tilde f^j\tilde f^tx,\beta) \subset \bigcap_{j=-n_0}^{n_0} \tilde f^{-j} \bar B(\tilde f^j\tilde f^tx,\beta) \\
& = W_{n_0}(\tilde f^tx,\beta) = \gamma W_{n_0}(\gamma^{-1}\tilde f^tx,\beta) \subset \gamma U(y_i),
\end{align*}
where in the second line we used $\Gamma\subset \Iso(\tilde V,d)$, $[\tilde f,\Gamma]=0$ and $(*) , n_0 \geq N(y_i)$.

\underline{Step 3.} (application of lemma \ref{lemmabowen1}) As a consequence of $(**)$, the set $\gamma E(y_i)$ is $(m(y_i),\delta/2)$-spanning for $\tilde f^t (F)$. We want to apply \ref{lemmabowen1}, so we define integers $0=t_0 < ... < t_r = n$ as follows.
\begin{enumerate}
 \item If $n\leq n_0$ take $r=1$ and $t_1=n$.
 \item If $n>n_0$, take $t_1=n_0$ and choose $i_1\in\{ 1,...,s \}, \gamma_1 \in \Gamma$ with $\tilde f^{t_1}(x)\in \gamma_1V(y_{i_1})$. Suppose now we have already choosen $t_1, ... , t_k$ with $t_k<n$ together with $i_1,...,i_k, \gamma_1,...,\gamma_k$.
 \begin{enumerate}
  \item If $t_k \geq n-n_0$, set $r=k+1$ and $t_r=n$.
  \item If $t_k < n-n_0$, set $t_{k+1}=t_k+m(y_{i_k})$ and choose $i_{k+1},\gamma_{k+1}$ with $\tilde f^{t_{i_{k+1}}}(x)\in \gamma_{k+1}V(y_{i_{k+1}})$.
\end{enumerate}
Eventually we are in case (a) and the process stops. Moreover we have $t_{r-2} <n-n_0 \leq t_{r-1} < n = t_r$ by $m(y_{i_{r-2}})\leq n_0$.
\end{enumerate}
Note that $t_{k+1}-t_k\leq n_0$ for $k=0,r-1$ and by $(**)$ the set $\gamma_k E(y_{i_k})$ is $(t_{k+1}-t_k, \delta/2)$-spanning for $\tilde f^{t_k}(F)$ for $k=1,...,r-2$. Choose $E_0,E_{r-1}$ to be $(n_0,\delta/2)$-spanning for $\bar B(x,\beta), \bar B(\tilde f^{t_{r-1}}x,\beta)$, respectively of minimal cardinality, so $E_0$ is also $(t_1-t_0,\delta/2)$-spanning for $F$ and $E_{r-1}$ is also $(t_r-t_{r-1},\delta/2)$-spanning for $\tilde f^{t_{r-1}}(F)$. Apply \ref{lemmabowen1} to
\[ E_0, ~ E_1 := \gamma_1 E(y_{i_1}) , ~..., ~ E_{r-2} := \gamma_{r-2} E(y_{i_{r-2}}), ~ E_{r-1} \]
and define
\[ \sqrt c :=  \sup_{y\in K} r_{n_0}(\bar B(y,\beta), \delta/2 ) ~~<\infty. \]
We obtain using the definition of $m(y_i)$ and $\sum_{k=1}^{r-2} m(y_{i_k}) \leq n-n_0 \leq n$ that
\begin{align*}
r_n(F,\delta) &\leq \# E_0 \cdot \left( \prod_{k=1}^{r-2} \# E_k \right) \cdot \# E_{r-1} \leq c \cdot \prod_{k=1}^{r-2} \# E(y_{i_k}) \\
&\leq c \cdot \prod_{k=1}^{r-2} e^{(a+\e)m(y_{i_k})} \leq  c \cdot e^{(a+\e)n}.
\end{align*}
Observe that $c$ depends only on $\delta, n_0, \beta$ and $n_0$ in turn is indepenent of $x,n$.
\end{proof}

Now we are able to prove the theorem.

\begin{proof}[Proof of \ref{bowen expansive}]
Let $E_n$ be a minimal $(n,\beta)$-spanning set for $K$ and let $\e,\delta >0$. Then
\[ K \subset \bigcup_{x\in E_n} \bar B_n(x,\beta) \]
and by \ref{lemma bowen} each of the sets in the above union can be $(n,\delta)$-spanned by using only $c e^{(a+\e)n}$ elements where $a=\hloc(\tilde f, \beta)$. Hence
\[ r_n(K,\delta) \leq \# E_n \cdot c e^{(a+\e)n} \leq r_n(K,\beta) \cdot c e^{(a+\e)n} \]
and
\[ \h(\tilde f, K, \delta)\leq \h(\tilde f, K, \beta)+ a+\e. \]
Letting $\e,\delta\to 0$, the claim follows using \ref{walters}.
\end{proof}

\section{Bounds for topological entropy }

\subsection{Lower Bound}

We need the following theorem stated in the book \cite{KH} of Katok and Hasselblatt on the topological entropy of minimal geodesics on Riemannian manifolds. For the convenience of the reader we will provide here a complete proof of the result, which differs from the one in \cite{KH} in small details. Recall the notation $p:\tilde M \to M$ for the universal cover of $M$ and
\begin{align*}
\tilde \M & = \{ v \in S\tilde M \mid c_v \mbox{ is a minimizing geodesic } \} \subset S\tilde M , \\
\M & = Dp(\tilde \M)\subset SM .
\end{align*}

\begin{theorem}\label{katok hasselblatt}
Let $(M,g)$ be a compact Riemannian manifold and $ \phi^t_\M$ be the geodesic flow $\phi ^t$ restricted to $\M\subset SM$. Then
$$ 
\h(\phi ^t_\M) \geq h(g).
$$
\end{theorem}

For the proof of \ref{katok hasselblatt} we need a lemma similar to lemma \ref{lemmabowen1}. Recall that $s_T(A, \delta)$ denotes the maximal cardinality of a $(T,\delta)$-separated subset of $A$.

\begin{lemma}\label{lemmabowen2}
Let $(V,d)$ be a metric space, $\phi^t : V \to V$ a continuous flow and $A\subset V$. For times $0 = t_0 < t_1 < \cdots < t_m = T$ and $\delta > 0$ we have
\begin{eqnarray*}
\prod\limits_{i = 1}^m s_{t_i - t_{i-1}} (\phi^{t_{i-1}} A,
\delta) \geq s_T (A, 2 \delta),
\end{eqnarray*}
\end{lemma}

\begin{proof}[Proof of \ref{lemmabowen2}]
Let $L$ be a maximal $(T, 2\delta)$-separated subset of $A$ and let $L_i$ be maximal $(t_i-t_{i-1}, \delta)$-separated subsets of $\phi^{t_{i-1}} (A)$ for $i=1,...,m$ . For $(x_1, \ldots, x_m) \in L_1 \times \cdots \times L_m$ set
\begin{align*}
& B(x_1, \ldots, x_m) := \\
& \{z \in L \mid ~ d (\phi^{t+ t_{i-1}}z, f^t x_i) \leq \delta ~~  \forall 1\leq i \leq m, t\in [0,t_i - t_{i-1}]\}.
\end{align*}
Since $L$ is $(T,2\delta)$-separated, the triangle inequality implies $\# B(x_1, \ldots, x_m) \leq 1$. Therefore, since the cardinalities of the $L_i$ are maximal implying that they are also $(t_i-t_{i-1}, \delta)$-spanning,
\begin{eqnarray*}
\# L = \# \left (\bigcup_{(x_1, ...., x_m) } B(x_1, ..., x_m) \right ) \leq  \prod\limits_{i=1}^m \# L_i.
\end{eqnarray*}
\end{proof}

\begin{proof}[Proof of \ref{katok hasselblatt}]
Fix $x \in \tilde M$, $\e>0$ and write
\[ \delta := \inj(M) >0, \quad h := h(g), \quad a := \sup_{y\in \tilde M} \vol B(y,2\delta), \quad  b :=\h(\phi^t_{SM}). \]

We have the following: there exists a sequence $T_k\to \infty$ such that
\[ \vol B(x,T_k + \delta/2)-\vol B(x,T_k) \geq e^{h(1-\e)T_k}, \]
for otherwise adding up the volume of the annuli $B(x,T_k + \delta(2) \setminus B(x,T_k)$ with $T_{k+1} = T_k + \delta/2$ starting at $T_0$ sufficiently large would yield that the exponential growth rate is less than  $h\cdot (1-\e)$.

Let $N_k$ be a maximal $2\delta$-separated set in the annulus $\bar B(x,T_k+ \delta/2) \setminus B(x,T_k)$, then we have for all $k\in \N$
\[ a \cdot \# N_k \geq \vol \left (\bigcup_{y\in N_k} B(y,2\delta)\right ) \geq \vol B(x,T_k + \delta/2)-\vol B(x,T_k) \geq e^{h(1-\e)T_k} .\]
For $y\in N_k$ let $c_y:[0,d(x,y)]\to\tilde M$ be a minimal geodesic segment with $c(0)=x$ and $c(d(x,y))=y$. Now, if $y_1,y_2\in N_k$ with $y_1\neq y_2$ we have
\[ d(c_{y_1}(T_k),c_{y_2}(T_k)) \geq d(y_1,y_2) -d(y_1,c_{y_1}(T_k))-d(y_2,c_{y_2}(T_k)) > \delta, \]
so the sets
\[ \tilde S_k :=\{\dot c_y(0) : y \in N_k\} \]
are $(T_k,\delta)$-separated w.r.t. the metric $d_1$ on $S\tilde M$, defined as
\[d_1(v,w) = \max_{t\in[0,1]} d(c_v(t),c_w(t)). \]
In $SM$ the sets $S_k := Dp(\tilde S_k)$ are $(T_k,\delta/2)$-separated. Define the decreasing sequence of compact sets
\[ \M_k := Dp \left\{v\in S\tilde M : c_v:[-\sqrt{T_k},\sqrt{T_k}]\to \tilde M \text{ is minimal}\right\}, \quad \bigcap_{k\in\N} \M_k = \M. \]
In order to find large separated sets in $\M$ we shall find them in the sets $\M_k$, observing that for $t\in [\sqrt{T_k}, T_k-\sqrt{T_k}]$ we have 
\[\phi^tS_k \subset \M_k. \]
Assume $k$ is large enough, s.th.
\[ s_{\sqrt{T_k}}(S_k,\delta/4) \leq e^{2b\sqrt{T_k}}, \quad \sqrt{T_k} \geq \frac{2b}{\e h} . \]
We apply lemma \ref{lemmabowen2} and obtain
\begin{align*}
& s_{T_k-\sqrt{T_k}}(\phi^{\sqrt{T_k}} S_k , \delta/4) \cdot s_{\sqrt{T_k}}(S_k,\delta/4) \geq s_{T_k}(S_k, \delta/2) \geq \# N_k \geq \frac{1}{a} e^{h(1-\e)T_k} \\
\Rightarrow \quad & s_{T_k-\sqrt{T_k}}(\phi^{\sqrt{T_k}} S_k , \delta/4) \geq \frac{1}{a} e^{h(1-\e)T_k - 2b\sqrt{T_k}} \geq \frac{1}{a} e^{h(1-2\e)T_k}.
\end{align*}
Let now 
\[  T \in (0, T_k-\sqrt{T_k} ] , \quad m_k = \left\lfloor \frac{T_k-\sqrt{T_k}}{T}\right\rfloor \in \N. \]
Applying lemma \ref{lemmabowen2} again gives
\begin{align*}
& \left( \prod_{i = 0}^{m_k-1} s_T(\phi^{iT + \sqrt{T_k}} S_k, \delta/8) \right) \cdot s_{T_k-\sqrt{T_k} - m_kT}(\phi^{m_kT + \sqrt{T_k}} S_k, \delta/8) \\
& \geq s_{T_k-\sqrt{T_k}} (\phi^{\sqrt{T_k}} S_k, \delta/4) \geq \frac{1}{a} e^{h(1-2\e)T_k} \\
\Rightarrow \quad & \prod_{i = 0}^{m_k-1} s_T(\phi^{iT + \sqrt{T_k}} S_k, \delta/8) \geq \frac{\frac{1}{a} e^{h(1-2\e)T_k}}{s_{T_k-\sqrt{T_k} - m_kT}(\phi^{m_kT + \sqrt{T_k}} S_k, \delta/8)} \\
& \geq \frac{\frac{1}{a} e^{h(1-2\e)T_k}}{s_T(SM, \delta/8)} \geq \frac{1}{a} e^{h(1-2\e)T_k - 2bT} ,
\end{align*}
where in the last step we assumed that $T$ is large, so that $s_T(SM, \delta/8) \leq e^{2bT}$. Hence one of the factors in the last product has to be ''large'', i.e. for some $i\in \{0,...,m_k-1\}$ we have
\[ s_T(\phi^{iT + \sqrt{T_k}} S_k, \delta/8) \geq \frac{1}{a} e^{\frac{h(1-2\e)T_k - 2bT}{m_k}} \geq \frac{1}{a} e^{h(1-2\e)T} e^{ - \frac{2bT}{m_k}} . \]
Note also that $\phi^{iT + \sqrt{T_k}} S_k \subset \M_k$, so when letting $k\to\infty$ while fixing $T$ and using $m_k\to\infty$, we find a $(T,\delta/8)$-separated set in $\M=\cap_k \M_k$ of cardinality at least
\[ \frac{1}{a} e^{h(1-2\e)T} \cdot \lim_{k\to\infty} e^{ - \frac{2bT}{m_k}} = \frac{1}{a} e^{h(1-2\e)T} . \]
This proves the theorem:
\[ \h(\phi^t_\M) \geq \h (\phi^t_\M, \delta/8) \geq h - 2\e . \]
\end{proof}

\subsection{Upper Bound for manifolds of hyperbolic type} \label{upbound}

Following Klingenberg \cite{Kl} we call a compact Riemannian manifold $(M,g)$  of hyperbolic type, if there exists a metric of strictly negative curvature $g_0$ on $M$. From now on we assume the existence of such $g_0$ on the compact Manifold $M$. When we lift objects such as $g, g_0$ from $M$ to the universal cover $\tilde M$ we will frequently denote them by the same letters. In the following we write $d$ for the metric on $\tilde M$ induced by $g$ and $d_{g_0}$ for the one induced by the background metric $g_0$. Due to the compactness of $M$ the two metrics on $\tilde M$ are equivalent, i.e. there exists a constant $C>0$ such that 
$$\frac{1}{C}d(p,q)\leq d_{g_0}(p,q) \leq Cd(p,q) \qquad \forall p, q \in \tilde M.$$  
We write $d_1$ for the metric on $S\tilde M$ defined by
$$
d_1(v,w) := \max_{t\in [0,1 ]} d(c_v(t), c_w(t))
$$
and $d_H(A,B)$ for the Hausdorff metric on sets $A,B\subset\tilde M$ w.r.t. $d$.

The following theorem is fundamental for the study of $\M$ in manifolds of hyperboldic type. It has been proven by Morse in dimension 2 and by Klingenberg in arbitrary dimensions.

\begin{theorem}[Morse lemma, cf. \cite{Kl} or \cite{Kn}] \label{theoremmorse}
Let $(M,g)$ be a manifold of hyperbolic type. Then there is a constant $r_0=r_0(g,g_0) >0$ with the following properties. \begin{itemize}
\item[(i)] If $c :[a,b] \rightarrow \tilde M$ and $\alpha : [ a_0, b_0] \rightarrow \tilde M$ are minimizing geodesic segments w.r.t. $g, g_0$, respectively, joining $c(a)=\al(a_0)$ to $c(b)=\al(b_0)$, then
$$d_H(c[a,b], \alpha[a_0, b_0]) \leq r_0.$$
\item[(ii)] For any minimizing $g$-geodesic $c :\R \to \tilde M$ there is a $g_0$-geodesic $\alpha :\R \to \tilde M$ and conversely for any $g_0$-geodesic $\alpha :\R \to \tilde M$ a minimizing $g$-geodesic $c :\R \to \tilde M$ with
$$
d_H(\alpha (\R) , c(\R )) \leq r_0 .
$$
\end{itemize}\end{theorem}

In this subsection we prove the following theorem stated in the introduction. As a consequence we immediately obtain corollary \ref{h=h(g) expansive} in the introduction using the results in section \ref{bowen bounds}.

\begin{theorem}\label{upper bound beta}
Let $(M,g)$ be a compact Riemannian manifold of hyperbolic type and $K \subset \tilde M$ a compact set in the universal cover $\tilde M$. 
Let 
$${\cal F}=SK \cap {\tilde {\cal M}},$$
where $SK=\pi^{-1}(K)$. Then there is some constant $\beta$ such that
$$ \h(\phi^t , {\cal F}, \beta) \leq h(g).$$
\end{theorem}

In order to prove the theorem, we construct spanning sets for $\cal F$. Let $K\subset {\tilde M}$ be a compact set  with $\diam K=a$. For $r >a$ consider
$$
K_r:=\{ z\in {\tilde M} \; | \; r - a \leq d(z,K) \leq r\}
$$ 
Let $K^\e, K_r^\e$ be minimal $\e$-spanning sets for $K,K_r$, respectively. For $y\in K^\e, z \in K_r^\e$, let $\alpha_{yz} : \mathbb{R} \to \tilde M$ be the $g_0$-geodesic connecting $y$ and $z$
such that $\alpha_{yz}(0) =y$ and  $\alpha_{yz}(d_{g_0}(y,z)) =z$. By the Morse lemma, there exists a minimizing $g$-geodesic $c_{yz} : \mathbb{R} \to \tilde M$ $r_0$-close to $\al_{yz}(\R)$. Set
\[ P_r:=\{ \dot c_{yz}(0) : y\in K^\e, z\in K_r^\e \} \subset \tilde \M . \]

\begin{lemma} \label{P_r spanning}
$P_r$ is a $(r-1, \beta)$-spanning set for ${\cal F}$ with respect to the metric $d_1$ where $\beta$ is given by
$\beta:=5r_0+(2C^2+1)\e$.
\end{lemma}

\begin{proof}[Proof of \ref{P_r spanning}]
Let $c:\R \to \tilde M$ be a minimizing $g$-geodesic with $c(0) \in K$. Then $c(r)\in K_r$ and we can choose $y\in K^\e,z\in K_r^\e$ with
$$
d(y,c(0)) \leq \e ,\quad d(z,c(r)) \leq \e.
$$
Let $\al$ be the $g_0$-geodesic connecting $c(0)$ and $c(r)$ parametrized such that $\al(0) = c(0) $ and $\al(d_{g_0}(y,z)) =  c(r)$. Using the convexity of the function $t \mapsto  d_{g_0}(\alpha(t), \alpha_{yz}(t))$ due to negative curvature we find
$$
d_{g_0}(\alpha (t) ,\alpha_{yz}(t)) \leq \max\{ d_{g_0}(c(0),y), d_{g_0}(c(r),z) \} \leq C\e \quad \forall ~ t \in [0, d_{g_0}(y,z)]. 
$$
Let $A=c[0,r]$ and $B =c_{yz}[0, r'] $ be the subsegment of $c_{yz}$ lying $r_0$-close to $\al_{yz}[0, d_{g_0}(y,z)]$ w.r.t. the $g$-Hausdorff metric $d_H$. Using the Morse lemma we find (omitting for the moment the intervals $[0, d_{g_0}(y,z)]$ for $\al,\al_{yz}$)
$$
d_H(A,B)\leq d_H(A,\al)+d_H(\al,\al_{yz})+d_H(\al_{yz}, B) \leq 2r_0+C^2\e . 
$$
By definition of the Hausdorff distance, for $t\in[0,r]$ there is some $t'\in \R$ with $d(c(t),c_{yz}(t'))\leq 2r_0+C^2\e$. 
Using the minimality of $c,c_{yz}$ we find with $d(c(0),c_{yz}(0))\leq r_0+\e$ that $|t-t'| \leq 3r_0 + (C^2+1)\e$ and hence
\[ d(c(t),c_{yz}(t)) \leq d(c(t),c_{yz}(t')) + d(c_{yz}(t'),c_{yz}(t)) \leq 2r_0+C^2\e + 3r_0 + (C^2+1)\e. \]
Therefore, taking $\beta:=5r_0+(2C^2+1)\e$ we obtain
$$
d_1(\dot{c}(t), \dot {c}_{yz}(t)) = \max \limits _{s\in [0,1]} d(c(t+s),c_{yz}(t+s)) \leq  \beta  \quad \forall t\in [0, r-1].
$$
\end{proof}

We can now prove the theorem.

\begin{proof}[Proof of \ref{upper bound beta}]
We have
\begin{align*}
& \# K_r^\e \leq C_\e \cdot \vol  B\left(x, r+a+\e/2\right) , \quad C_\e := \left( \inf_{y\in M}\vol B(y,\e/2) \right)^{-1}, \\
\Rightarrow \quad & \# P_r \leq \# K^\e \cdot \# K_r^\e \le \# K^\e \cdot C_\e \cdot  \vol  B\left(x, r+a+\e/2\right) .
\end{align*}
Hence
\begin{align*}
& \h(\phi^t , {\cal F}, \beta) \leq \varlimsup_{r\to \infty }\frac{1}{r-1}\log \# P_r \leq \varlimsup_{r\to \infty }\frac{1}{r-1}\log  \vol  B \left(x, r+a+ \e/2 \right) \\
&= \lim_{r \rightarrow \infty }\frac{ r+a+\e/2} {r-1}\frac{1}{r+a+\e/2}\ \log  \vol  B\left(x, r+a+\e/2 \right )  = h(g) .
\end{align*}
\end{proof}

\section{The two-dimensional case}

We use the notation introduced at the beginning of section \ref{upbound}. Morse \cite{morse} studied the structure of minimal geodesics in the universal cover $\tilde M$ (called ''class A geodesics'' there), where $M=\tilde M/\Gamma$ is a closed orientable surface of genus $\geq 2$. Apart from the Morse lemma in section \ref{upbound}, which is valid in any dimension, the assumption $\dim M=2$ provides additional information since in $\tilde M$ the minimizing geodesics intersect quite easily. As a background metric for $M$ we can choose by the uniformisation theorem a metric of constant negative curvature $-1$ and we use for $\tilde M$ the Poincar\'e model given by
\[ \tilde M=\{ z\in\C: |z|<1 \}, \quad (g_0)_z=\frac{4}{(1-|z|^2)^2}\skp_{euc}. \]
This model has a simple boundary at infinity, namely $\tilde M(\infty)=S^1$. Using the Morse lemma, for pairs $\xi_-,\xi_+\in S^1$ with $\xi_-\neq \xi_+$ we distinguish the minimal $g$-geodesics lying in bounded distance from the $g_0$-geodesic in $\tilde M$ joining $\xi_-,\xi_+$. Write
\begin{align*}
& c(\pm\infty):=\lim_{t\to\pm\infty}c(t) = \xi_\pm  \qquad \text{(the limit in the euclidean sense in $\C$)} \\
& B := \{ \xi= (\xi_-,\xi_+): \xi_-, \xi_+ \in S^1, \xi_-\neq \xi_+ \} = \tilde M(\infty)\times\tilde M(\infty)-\diag, \\
& \tilde \M_\xi := \{ \dot c(0) \mid \text{$c:\R\to \tilde M$ is an arc-length $g$-minimal with $c(\pm\infty)=\xi_\pm$} \} ,
\end{align*}
then $\tilde \M = \cup_{\xi \in B} \tilde\M_\xi$ and each class $\tilde\M_\xi$ is non-empty. In the sequel \emph{minimal} refers to $g$-minimizing arc-length geodesics $c:\R\to \tilde M$.

\subsection{Structure of the minimals}

\begin{definition}\label{M^pm}
For $v\in \tilde \M$ let $\tilde M^+(v),\tilde M^-(v)$ be the open connected components (half discs) of $\tilde M-c_v(\R)$, where $\tilde M^+(v)$ contains $\pi v + t \cdot i v$ for small $t>0$. For $\xi \in B$ set
\begin{align*}
\tilde\M_\xi^+ & := \left \{v \in \tilde\M_\xi ~\big|~  \forall w\in\tilde\M_\xi : \quad  \pi w \in c_v(\R) \quad\Rightarrow\quad c_w[0,\infty) \subset \overline{\tilde M^-(v)} ~ \right \}, \\
\tilde\M_\xi^- & := \left \{v \in \tilde\M_\xi ~\big|~  \forall w\in\tilde\M_\xi : \quad  \pi w \in c_v(\R) \quad\Rightarrow\quad c_w[0,\infty) \subset \overline{\tilde M^+(v)} ~ \right \}, \\
\tilde\M_\xi^0 & := \tilde\M_\xi^+\cup \tilde\M_\xi^- , \quad \tilde \M^0 := \bigcup_{\xi\in B} \tilde\M_\xi^0 \subset S\tilde M, \quad \M^0 := Dp(\tilde\M^0) \subset SM.
\end{align*}
\end{definition}

\begin{remark}\label{remark M^pm} \begin{itemize}
\item[(i)] The sets $\tilde\M_\xi^\pm$ are never empty. In fact, the intersection $\tilde\M_\xi^- \cap \tilde\M_\xi^+$ contains the velocity vectors of the bounding geodesics of $\tilde\M_\xi$ (cf. theorem 8 in \cite{morse}).

\item[(ii)] It is easy to see that no two geodesics from $\tilde\M_\xi^+$ (resp. $\tilde\M_\xi^-$) intersect transversely. We shall refer to this as the graph property of $\tilde\M^\pm_\xi$.

\item[(iii)] The sets $\tilde\M_\xi^\pm$ and hence $\tilde\M^0$ and $\M^0$ are closed and $\phi^t$-invariant.
\end{itemize}\end{remark}

By (ii) in \ref{remark M^pm} the sets $\tilde\M^0_\xi$ have a simple structure in $\tilde M$, so when calculating $\h(\phi_\M^t)$ we would like to stick to $\M^0$. For this it is important that $\M^0$ is ''sufficiently large''. Let $\Om\subset SM$ denote the non-wandering set of $\phi^t$ restricted to $\M$. The following proposition is the key observation to obtain $\h(\phi^t_\M)=h(g)$ in the two-dimensional case.

\begin{prop}\label{M^0 recurrent}
$\M^0\subset SM$ contains the non-wandering set $\Om$ of $\phi_\M^t$.
\end{prop}

\begin{proof}
Let $v\in Dp^{-1}(\Om)\cap \tilde\M_\xi$ and $U_n=B(v,1/n)\cap \tilde\M\subset S\tilde M$ for $n\in\N$. By definition of $\Om$ there exists $\gamma_n\in\Gamma-\{\id\}$ and $t_n>0$ such that $ D\gamma_n\phi^{t_n} U_n \cap U_n \neq \emptyset$. In particular there is some $v_n \in U_n$ such that $w_n := D\gamma_n \phi^{t_n} v_n \in U_n$. Assume $v\notin\tilde\M_\xi^0$, so there are two minimals $c^\pm :\R \to \tilde M$ in $\tilde\M_\xi$ with $c^\pm(0)=c_v(t^\pm)$ and $c^\pm(0,\infty)\subset \tilde M^\pm(v)$.

First suppose $c_{v_n}(\infty)=c_{w_n}(\infty)=c_v(\infty)=\xi_+\in\tilde M(\infty)$ for some $n$. Then $\xi_+=c_{w_n}(\infty) =\gamma_n c_{v_n}(\infty)=\gamma_n \xi_+$, so $\xi_+$ is the point at $+\infty$ for some periodic minimal axis of $\gamma_n$. If $c_v$ is itself periodic, theorems 10 and 13 in \cite{morse} show that in fact there are no minimal geodesics in $\tilde \M_\xi$ intersecting $c_v$ transversely, i.e. $v\in \tilde\M_\xi^0$. If $c_v$ is not periodic, it is asymptotic to some periodic minimal $c_0$ in $+\infty$, approaching its limit from ''below'' (i.e. from $\tilde M^-(c_0(\R))$), say, by theorem 10 in \cite{morse}. Now $c^+$ is also asymptotic in $+\infty$ to that same minimal $c_0$ (theorem 13 in \cite{morse}). But two asymptotic minimals in $\tilde M$ cannot intersect transversely (theorem 6 in \cite{morse}), so we obtain a contradiction.

Assume now that $c_{v_n}(\infty) \neq c_{w_n}(\infty)$ for all $n\in\N$. Interchanging $v_n, w_n$ and maybe taking a subsequence, we may assume that $v_n\to v$ and $c_{v_n}(\infty)\neq \xi_+$ for all $n$. Moreover we can assume that the $c_{v_n}(\infty)$ lie in one connected component of $\tilde M(\infty)-\{ \xi_-,\xi_+ \}$, say  $c_{v_n}(\infty) \in \tilde M(\infty) \cap \overline{\tilde M^+(v)}$. Now, by $\dot c_{v_n}(t^+)\to \dot c_v(t^+)$ and the assumptions on the points at infinity of $c_{v_n},c^+$, there have to be two intersections of $c_{v_n},c^+$ for large $n$, contradicting the minimality of both geodesics. 
\end{proof}

\subsection{Entropy in strips of finite width}

In this section we will show that the local entropy of the geodesic flow in the non-wandering set $\Om\subset \M^0$ of $\phi^t_\M$ is vanishing. We work in the universal cover and write $\tilde \Om := Dp^{-1}(\Om)\subset S\tilde M$ for the lifted non-wandering set of $\phi^t_\M$. Recall
\begin{align*}
& d_1(v,w)=\max_{t\in[0,1]}d(c_v(t),c_w(t)) \quad v,w\in S\tilde M , \\
& Z_\beta(v) = \{ w \in \tilde \Om : d(c_v(t),c_w(t)) \leq \beta ~ \forall t\in\R \} \subset S\tilde M, \quad v\in\tilde \Om.
\end{align*}

\vspace*{20pt}

\begin{prop}\label{entropy-exp}
For any $v_0 \in \tilde \Om$ and any $\beta>0$ we have
\[ \h(\phi^t_{\tilde\Om}, Z_\beta(v_0))=0. \]
Hence the geodesic flow restricted to $\tilde \Om$ is $\beta$-entropy-expansive for any $\beta>0$.
\end{prop}

\begin{proof}
Fix $v_0\in\tilde \Om, \beta>0$ and some small $\delta >0$. By \ref{M^0 recurrent} we find $\xi\in B$ with $v_0\in\tilde \M_\xi^0$ and hence $Z_\beta(v_0) \subset \tilde\M_\xi^0$. We shall prove that $(T-1,2\delta)$-spanning sets $E$ of minimal cardinality for $Z_\beta(v_0)\cap \tilde\M_\xi^+$ have cardinality growing at most linearly in $T$. The same arguments work for $Z_\beta(v)\cap \tilde\M_\xi^-$ and hence give the proposition.

Write
\begin{align*}
A & := Z_\beta(v_0)\cap \tilde\M_\xi^+ , \\
K_T & :=\{x\in\tilde M: d(x,c_{v_0}[0,T])\leq\beta\} , \\
\TT(c_v,\delta) & := \{ x \in \tilde M : d(x,c_v(\R)) < \delta \} , \quad v\in A.
\end{align*}
The sets $A,K_T$ are compact.

\underline{Step 1.} For $\delta>0$ and $v,w\in A$ with $c_w[0,T]\subset \TT(c_v,\delta/3)$ there exists
 $s_0 \in \R$  such that
\[ d(c_v(t+s_0),c_w(t)) \leq \delta \quad \forall t\in[0,T]. \]

Proof. By assumption for any $t\in [0,T]$ there is some $s(t)\in \R$ with
\[d(c_v(s(t)),c_w(t))\leq \delta_0 := \delta/3. \]
Using the minimality of $c_v,c_w$ one finds
\[ s(t)-s(0) \leq 2\delta_0 +t , \quad t \leq 2\delta_0 + s(t)-s(0). \]
Hence with $s_0:= s(0)$ we have
\begin{align*}
 d(c_v(t+s_0),c_w(t))  & \leq  d(c_v(t+s(0)),c_v(s(t))) + d(c_v(s(t)),c_w(t)) \\
&  \leq |s(t)-s(0)-t|+\delta_0 \leq 3\delta_0 = \delta.
\end{align*}

\underline{Step 2.} Let $F(v,\delta) := \{ \phi^{j\delta}v : j\in\Z, |j| \leq 2(1+\beta/\delta) \}$ for $v\in A$. Then for $v,w\in A$ with $c_w[0,T]\subset \TT(c_v,\delta/3)$ there exists  $v_j =\phi^{j\delta}v\in F(v,\delta)$ such that
\[ d(c_{v_j}(t),c_w(t)) \leq 2\delta \quad \forall t\in[0,T]. \]

Proof. Let $s_0$ be as in step 1 and $j\in \Z, r\in[0,\delta)$ with $s_0 = j\delta + r$. Then
\begin{align*}
 d(c_v(t+j\delta) , c_w(t)) &  \leq d(c_v(t+j\delta) , c_v(t+s_0)) + d(c_v(t+s_0) , c_w(t)) \\
& \leq | t+j\delta - t-s_0 | + \delta \leq 2\delta. 
\end{align*}
By definition of $A$ we have $d(\pi v,\pi w)\leq 2\beta$ and hence again by step 1
\[ |s_0| = d(c_v(s_0),\pi v) \leq d(c_v(s_0),\pi w) + d(\pi w,\pi v) \leq \delta + 2\beta, \]
showing
\[ |j| \leq \frac{|s_0|+|r|}{\delta} \leq \frac{\delta + 2\beta+\delta}{\delta} = 2(1+\beta/\delta). \]

\underline{Step 3.}
\[ \h(\phi^t,A)=0 . \]

Proof. Consider the family of (oriented) unparametrised curves
\[ \A : = \{c_v(\R) \subset \tilde M : v\in A \}. \]
$\A$ is ordered by the graph property of $\tilde\M_\xi^+$ ($c<c'$ iff $c'\subset \tilde M^+(c)$) and we construct a sequence of geodesics $c_1< ... < c_n< ...$ in $\A$. 
By closedness of $A$ we find a $<$-smallest geodesic $c_1$ in $\A$. If $c_1,...,c_n$ are already chosen, take $c_{n+1}\in \A$ to be the $<$-smallest geodesic $c_{n+1}>c_n$, such that the compact segment 
$c_n\cap K_T$ is not entirely contained in the open tube $\TT(c_{n+1},\delta/3)$. By construction, there is some $p_n\in c_n\cap K_T$ with $d(p_n,c_{n+1}(\R))\geq \delta/3$, hence the upper open half disc
\[ D_n := \tilde M^+(c_n) \cap B(p_n,\delta/3) \]
lies in the open strip between $c_n<c_{n+1}$ (for $\delta$ small, s.th. $c_n$ does not return to $D_n$ by minimality). 
Moreover all half discs $D_n$ are contained in a $\delta/3$-neighborhood of $K_T$ and disjoint, since the $c_i$ are ordered. 
As the volume of $D_n$ is bounded from below by some constant $C(\delta)$ using standard comparison theorems and the compactness of $M$, 
and the volume of the $K_T$-neighborhood is finite, growing linearly with $T$, the above construction stops at some finite $N(T)$, again $N(T)$ growing at most linearly. 
On the other hand, by construction for any $c\in\A$ we find some $i\in\{1,...,N(T)\}$ such that $c\cap K_T \subset \TT(c_i,\delta/3)$. Choose the parameterization of the $c_i$ such that $v_i := \dot c_i(0) \in A$. Now by step 2 the set
\[ E(T,\delta) := \bigcup_{i=1}^{N(T)} F(v_i,\delta) \]
is $(T-1,2\delta)$-spanning for $A$ w.r.t. $d_1$ with cardinality
\[ \# E(T,\delta) = N(T) \cdot \# F(v_i , \delta) = N(T) \cdot ( 4(1+\beta/\delta) + 1). \]
Hence for any $\delta>0$ we have
\[ \h(\phi^t,A,2\delta) \leq \lim_{T\to\infty}\frac{\log \#E(T,\delta)}{T-1} = \lim_{T\to\infty}\frac{\log N(T)}{T-1} = 0 \]
and by letting $\delta\to 0$, the claim follows.
\end{proof}

We now have the following result.

\begin{theorem}
If $M$ is a closed orientable surface with genus $\geq 2$, then
\[ \h(\phi^t_\M) = h(g) \]
for any Riemannian metric $g$ on $M$.
\end{theorem}

\begin{proof}
By \ref{katok hasselblatt} we have $\h(\phi^t_\M)\geq h(g)$. To show the reverse inequality take any compact set $K\subset \tilde M$ 
with $p(K)=M$. By \ref{upper bound beta} there is some $\beta>0$ with
\[ \h(\phi^t,\tilde \Om\cap SK, \beta) \leq \h(\phi^t,\tilde \M\cap SK, \beta) \leq h(g). \]
Using \ref{bowen expansive} and \ref{entropy-exp} we find
\begin{align*}
\h(\phi^t_\Om) & \leq \h(\phi^t,\tilde \Om \cap SK, \beta) + \hloc(\phi^t,\tilde \Om \cap SK, \beta) \\
& = \h(\phi^t,\tilde \Om \cap SK, \beta) \leq h(g).
\end{align*}
But since $\Om$ is has full measure w.r.t. any invariant probability on $\M$, we find (cf. 8.6.1 (ii) in \cite{W})
\[ \h(\phi_\M^t) = \h(\phi_{\Om}^t) \leq h(g). \]
\end{proof}

\begin{remark}\begin{itemize}
\item[(i)] The theorem is trivial for the 2-sphere $M=S^2$ and also holds in the case where $M$ is the 2-torus $M=T^2$. This can be shown using the same ideas presented in this paper and is implemented in \cite{glasmachers}. If $M$ is non-orientable, the theorem holds for the orientable double cover $\hat M$ of $M$ and $M,\hat M$ have the same universal cover. Hence:

For any closed surface $M$ with any Riemannian metric $g$ we have 
\[ \h(\phi^t_\M) = h(g) \]
as stated in the introduction.

\item[(ii)] 
 Theorem 4.5. also holds if one replaces the Riemann metric $g$ by a Finsler metric $F$ (not necessarily reversible).
The Morse lemma only requires that the two norms $F, F_0 = \sqrt{g_0}$ are equivalent.  The  volume entropy of $F$ can be defined by
$$
h(F) = \lim_{r \to \infty} \frac {\log \vol_{g_0} B(p,r)}{r} 
$$  where $ \vol_{g_0}$ is the $g_0$-volume (in  fact one could take the lift of an arbitrary Riemannian metric to compute the volume)
and  $B(p,r)$ is the ball defined by the endpoints of Finsler geodesics rays 
of length $\le r$ initiating form $p$.
The  result \ref{katok hasselblatt} of Katok and Hasselblatt holds also in this setting, just as the arguments for \ref{upper bound beta}. Arguing along these lines
yields Theorem 4.5 in the Finsler case.  This has implications for Tonelli Lagrangian systems, as on high enough energy levels the arising Euler-Lagrange flow is a reparametrisation of a Finsler geodesic flow by Maupertuis' principle.
\end{itemize}\end{remark}


\begin{thebibliography}{1234}

\bibitem{Bo} R. Bowen -- \emph{Entropy-expansive maps}. Trans. Amer. Math. Soc. 164 (1972), 323-331.

\bibitem{FM} A. Freire,  R. Ma\~n\'e -- \emph{On the Entropy of the Geodesic Flow in Manifolds Without Conjugate Points}. Invent. math. 69 (1982), 375-392.

\bibitem{glasmachers} E. Glasmachers -- \emph{Characterization of Riemannian metrics on $T^2$ with and without positive topological entropy}, PhD thesis, Ruhr-Universit\"at Bochum (2007).

\bibitem{KH}  A. Katok, B. Hasselblatt -- \emph{Introduction to the Modern Theory of Dynamical Systems}. Cambridge University Press (1995).

\bibitem{Kl} W. Klingenberg -- \emph{Geod\"atischer Fluss auf Mannigfaltigkeiten vom hyperbolishen Typ}. Inventiones math. 14 (1971), 63-82.

\bibitem{Kn} G. Knieper -- \emph{ Hyperbolic Dynamics and Riemannian Geometry}, in: Handbook of Dynamical Systems 1A, Elsevier Science, Editors: B. Hasselblatt and A. Katok (2002), 453-545.

\bibitem{Ma} A. Manning -- \emph{Topological entropy for geodesic flows}. Annals of Math. 110 (1979), 567-573.
  
\bibitem{morse}  M. Morse -- \emph{A fundamental class of geodesics on any closed surface of genus greater than one}. Trans. Amer. Math. Soc. 26 (1924), 25-60.


\bibitem{W} P. Walters -- \emph{An introduction to ergodic theory}. Graduate texts in Mathematics 79, Springer-Verlag (1982).
  
\end{thebibliography}
\end{document}